\documentclass[reqno]{amsart}
\usepackage[english]{babel}
\usepackage{graphicx}
\usepackage{textcomp}
\usepackage{amssymb}
\usepackage{color}
\newtheorem{definition}{Definition}[section]

\newtheorem{theorem}{Theorem}[section]
\newtheorem{corollary}{Corollary}[section]
\newtheorem{remark}{Remark}[section]

\newtheorem{proposition}{Proposition}[section]

\usepackage{framed}

\author{Olga~S.~Rozanova}
\address{Mathematics and Mechanics Faculty, Moscow State University, Moscow
119992, Russia.} \email{rozanova@mech.math.msu.su}

\title[Frozen and almost frozen structures]
{Frozen and almost frozen structures in the compressible rotating
fluid}



\date{ }

\begin{document}

\begin{abstract} We study a possibility of existence of localized
two-dimensional structures, both smooth and non-smooth, that can
move without significant change of their shape in a leading stream
of compressible barotropic fluid on a rotating plane.
\end{abstract}

\maketitle




\section{Bidimensional model of compressible fluid}\label{Sec1}

We consider the  system of barotropic gas dynamics in 2D on the
rotating plane,
\begin{equation}\label{2d_U}
 \varrho(\partial_t {\bf U} +  ({\bf U}\cdot \nabla ){\bf U} +
 {\mathcal L}{\bf U}) + \nabla P(\rho) = 0,\quad P=C
\rho^\gamma, \quad C
>0,
\end{equation}
\begin{equation}\label{2d_rho}
\partial_t \varrho +  {\rm div} ( \varrho {\bf U}) =0,
\end{equation}
for density $\varrho,$ vector of velocity ${\bf U}$ and pressure
$P$, $t\ge 0$, $x\in{\mathbb R}^2$.  Here $\mathcal L = l L$,
$\, L = \left(\begin{array}{cr} 0 & -1 \\
1 & 0
\end{array}\right)$,
 $l$ is the Coriolis parameter
assumed to be a positive constant,  $\gamma\in (1,2)$ is the
adiabatic exponent.

Under suitable boundary conditions system (\ref{2d_U}),
(\ref{2d_rho}) implies conservation of mass, momentum and total
energy.

Many models  of ocean, atmosphere and plasma are approximately two
dimensional.   In particular, in \cite{Obukhov} a procedure of
averaging over the hight in a three-dimensional model of atmosphere
consisting of compressible rotating polytropic gas was proposed (see
also \cite{Pedloski}).

Let us introduce a new variable $\pi=P^{\frac{\gamma-1}{\gamma}}.$
 For  $\pi(t,x)$ and ${\bf U}(t,x)$
we obtain the following system:
$$
\partial_t {\bf U} +  ({\bf U}\cdot \nabla ){\bf U} +
l L {\bf U} +  c_0\,\nabla \pi =0,\quad
\partial_t \pi +  (\nabla \pi \cdot  {\bf U}) +
(\gamma-1) \pi \, {\rm div}  {\bf U}=0,
$$
with $c_0= \frac{\gamma}{\gamma-1} C^{\frac{1}{\gamma}}.$ Then we
change the coordinate system in such a way that the origin of the
new system ${\bf x}=(x_1,x_2)$ is located at a point ${\bf
X}(t)=(X_1(t),X_2(t))$ (here and below we use the lowercase letters
for $\bf x$ to denote the local coordinate system). Now ${\bf U}=
{\bf u}+{\bf V},$ where
 ${\bf V}(t)=(V_1(t),V_2(t))=(\dot X_1(t),\dot X_2(t))$.
Thus, we obtain a new system
\begin{eqnarray}\label{sys1}
\partial_t {\bf u} + ({\bf u} \cdot \nabla ){\bf u} +
\dot {\bf V}+ l \, L \, ({\bf u}+{\bf V})+ c_0\, \nabla \pi = 0,\\
\quad
\partial_t \pi +  (\nabla \pi \cdot {\bf u}) + (\gamma-1)\,\pi\,{\rm div}\,{\bf
u}\, =\,0.\label{sys2}\end{eqnarray}
 Given a vector ${\bf V}$, the
trajectory can be found by integration from the system $\dot
X_i(t)=V_i(t),\, i=1,2.$

In our previous papers  \cite{RYH2010}, \cite{RYH2012} we used this
approach to find a position of atmospherical vortex, associated with
a tropical cyclone.

\section{Local and bearing fields separation}\label{Sec3}

We assume that the pressure field $\pi(t,x_1,x_2)$ can be separated
into two part as $\pi=\pi_0+\pi_1$, where $\pi_1$ is somewhat
stronger, however  more uniform than $\pi_0$. We call $\pi_0$ the
local field and $\pi_1$ the  bearing field.

Then we obtain from (\ref{sys1}), (\ref{sys2})
$$\left[\partial_t {\bf u} +
({\bf u} \cdot \nabla ){\bf u} + {\mathcal L} {\bf u} + c_0\, \nabla
\pi_0 \right]+ \left[\dot {\bf V}+ {\mathcal L} \, {\bf V}+ c_0\,
\nabla \pi_1\right] = 0,\quad \qquad\phantom{\pi_1\,{\rm div}\,{\bf
u}\,) =\,0}
$$
$$
\left[\partial_t \pi_0 +  (\nabla \pi_0 \cdot {\bf u}) +
(\gamma-1)\,\pi_0\,{\rm div}\,{\bf u}\right]\,+\, \left[\partial_t
\pi_1 +(\nabla \pi_1 \cdot {\bf u}) + (\gamma-1)\,\pi_1\,{\rm
div}\,{\bf u}\,\right] =\,0.
$$
If we assume that we can find the couple $({\bf u}, \pi_0)$ from the
system
\begin{equation}
\label{u-nonsteady}
\partial_t {\bf u} + ({\bf u} \cdot \nabla ){\bf u} +  {\mathcal L}
\,{\bf u} + c_0\, \nabla \pi_0 =Q(t,x),
\end{equation}
\begin{equation}\label{pi_0-nonsteady}
\partial_t \pi_0 +  (\nabla \pi_0 \cdot {\bf u}) +
(\gamma-1)\,\pi_0\,{\rm div}\,{\bf u}=\,0,  \end{equation} with a
certain function $Q(t,x)$,
 then we
get a  linear equation for $\pi_1$,
\begin{equation}\label{pi1}
\partial_t \pi_1 +  (\nabla \pi_1 \cdot {\bf u})+ \pi_1 {\rm div}{\bf u}= 0,
\end{equation}
which  can be solved for any initial condition $\pi_1(0,{\bf x})$.
Further, (\ref{u-nonsteady}) and (\ref{sys1}) imply
\begin{equation*}
\dot {\bf V}(t)+ {\mathcal L} {\bf V}(t) +c_0 \nabla \pi_1(t,{\bf
x})=- Q(t,x).
\end{equation*}
Now we set
\begin{equation}
\label{Q} Q =-c_0\left[\nabla \pi_1(t,{\bf x})-
\nabla\pi_1(t,{\bf 0})\right].
\end{equation}

Thus, we associate the couple $({\bf u}, \pi_0)$ with the local
field and the couple $({\bf V}, \pi_1)$ with the bearing field. As
we can see, the couple $({\bf u}, \pi_0)$ is independent of the
bearing field "up to the function $Q$." If the solution $({\bf u},
\pi_0)$ to system (\ref{u-nonsteady}), (\ref{pi_0-nonsteady}) is
found, we can find $({\bf V}, \pi_1)$ from linear equations.

If $Q=0$, then the bearing field does not influence on the local
field and in this sense we will talk on a complete separation of he
bearing and local fields. Evidently, this will be only if $\pi_1$ is
linear with respect to the space variables. If $|Q|<\delta$ for
sufficiently small $\delta>0$, we can talk about a "$\delta$-
approximate" separation of fields, $Q$ plays a role of discrepancy.
This discrepancy is a measure of separability of the local and
bearing fields.

The position of the center of the moving coordinate system  can be
found  from the following equation:
\begin{equation}
\label{xQ_exact}
\ddot {\bf X}(t)+ {\mathcal L} \dot{\bf X}(t) +c_0 \nabla
\pi_1(t,{\bf x})\Big|_{{\bf x}=0}= 0.
\end{equation}
As follows from the computer modeling made in \cite{RYH2012}, even
for real meteorological data the position of center of tropical
cyclone found by means of equation (\ref{xQ_exact}) is quite
accurate.



\section{Steady nonhomogeneous incompressible flow}\label{frozen}

  We  look for a solution of the local field with special properties,
namely, a steady divergence free solution. If the discrepancy $Q=0$,
then this solution can be considered as a  "frozen pattern" into a
leading stream. If the discrepancy is small, we can talk only on an
"almost frozen pattern", since the right-hand side in equation
(\ref{u-nonsteady}), that depends on the properties of the bearing
field, influences the solution.


Thus, let us assume that
$${\bf u} \mbox{\quad does not depend of}\quad
t,\quad {\bf u}(0)=0 \mbox{\quad and \quad} {\rm div \,\bf u} =0.$$
This means that there exists a stream function  $\Phi(x_1,x_2)$ such
that
\begin{equation*}\label{u}{\bf u}=\nabla_\bot \Phi = (\Phi_{x_2}, - \Phi_{x_1}).\end{equation*}
Equations (\ref{pi_0-nonsteady}) and (\ref{u-nonsteady}) result
\begin{equation}\label{constraint}(\nabla\pi_0 \cdot \nabla_\bot \Phi) = 0,\end{equation}
\begin{equation}
\label{Phi_pi_0} (\nabla_\bot \Phi \cdot \nabla)\nabla_\bot \Phi +
lL \nabla_\bot \Phi
+ c_0 \nabla \pi_0=0.
\end{equation}

 We take the inner product
of (\ref{Phi_pi_0}) and $\nabla_\bot \Phi$ and get
\begin{equation}\label{master}
 \Phi_{x_1 x_2} \,( ( \Phi_{x_2})^2-(\Phi_{x_1})^2))
  =\,
 \Phi_{x_1}\, \Phi_{x_2}\,(\Phi_{x_2 x_2}-
\Phi_{x_1 x_1}).
\end{equation}
 The solution of (\ref{master}) have to satisfy the identity
\begin{equation}
\label{cond_pi} \nabla \times ((\nabla_\bot \Phi \cdot
\nabla)\nabla_\bot \Phi )=0.
\end{equation}

Equations (\ref{master}) and (\ref{cond_pi}) are equivalent to
\begin{equation}
\label{J_master} J(\Phi,|\nabla \Phi|^2)=0,
\end{equation}
and
\begin{equation}
\label{J} J(\Phi,\Delta \Phi)=0,
\end{equation}
respectively, where $J$ is the Jacobian. The stream function $\Phi$
has to satisfy both (\ref{J_master}) and (\ref{J}), and for smooth
$\Phi$ the function $\pi_0$ can be restored up to a constant as
$$
\pi_0\,=\, -\, \frac{1}{ c_0}\, \left[\int(\Phi_{x_2} \Phi_{x_1
x_2}-\Phi_{x_1} \Phi_{x_2 x_2}+  l \Phi_{x_1})\,dx_1
+\right.$$$$\left. \int(-\Phi_{x_2} \Phi_{x_1 x_1}+\Phi_{x_1}
\Phi_{x_1 x_2}+ l \Phi_{x_2})\,dx_2 \right]=
$$
\begin{equation}\label{pi_formula}
- \frac{1}{c_0}\, l\Phi + \left[\int(\Phi_{x_2} \Phi_{x_1
x_2}-\Phi_{x_1} \Phi_{x_2 x_2})\,dx_1 + \int(-\Phi_{x_2} \Phi_{x_1
x_1}+\Phi_{x_1} \Phi_{x_1 x_2})\,dx_2 \right].
\end{equation}

There are two evident classes of solution to (\ref{J_master}) and
(\ref{J}): $$\Phi=\bar\Phi(x_1^2+x_2^2)$$ and
$$\Phi=\bar\Phi(x_i),\,i=1,2,$$ with arbitrary smooth function of one
variable $\bar\Phi$. In particular, $\bar\Phi$ can be compactly
supported.

The first case corresponds to a steady vortex.  In the
meteorological model this  pattern can be associated with a tropical
cyclone in the mature stage of development.

The second case corresponds to a shear flow and can be associated
with an atmospheric front.

\begin{remark} If $\pi_0 = \rm const$, then the condition
(\ref{constraint}) is eliminated and  the problem can be reduced to
the solution of the Euler equations. Even in this case possible
steady solutions can be very complicate \cite{Wu_Ov_Zab},
\cite{Yang_Kubota}, \cite{Jia}.
\end{remark}

\begin{remark}
Equations (\ref{J_master}) and (\ref{J}) imply the Dubreil-Jacotin
equation \cite{Dubreil-Jacotin}
$$
\Delta \Phi + \frac{\pi_0'(\Phi)}{2(\gamma-1)\pi_0(\Phi)}|\nabla
\Phi|^2=F(\Phi)
$$
if we take into account that $\pi_0=\pi_0(\Phi)$. The last property
follows from (\ref{constraint}). The arbitrary functions
$\pi_0(\Phi)$ and $F(\Phi)$ give initial distribution of density and
vorticity.
\end{remark}

\subsection{Algorithm of solution, the smooth case}

\begin{theorem} \label{discrepancy}
Let $\Phi(x_1, x_2)$ be a smooth solution to the system
\begin{equation}\label{ei_Phi}
|\nabla \Phi|^2=G(\Phi),
\end{equation}
\begin{equation}\label{delta_Phi}
\Delta \Phi=R(\Phi),
\end{equation}
with a differentiable function $G$ and integrable function $R$. Then
$\Phi$ solves (\ref{J_master}),(\ref{J}), and therefore it is a part
of solution to the system (\ref{constraint}), (\ref{Phi_pi_0}).
\end{theorem}
\begin{proof}
Equations (\ref{ei_Phi})  and (\ref{delta_Phi}) means that $|\nabla
\Phi|^2$ and $\Phi$, $\Delta \Phi$ and $\Phi$, respectively, are
functionally dependent,  this results (\ref{J_master}) and
(\ref{J}).
  Thus, $\Phi$ solves the system
(\ref{constraint}), (\ref{Phi_pi_0}) together with $\pi_0$ found by
(\ref{pi_formula}).  Namely, taking into account (\ref{ei_Phi}) and
(\ref{delta_Phi}) we obtain
$$
\pi_0\,= -\, \frac{1}{2c_0}\, \left[\int(|\nabla \Phi|^2\,-
\,2R_1(\Phi)+2l \Phi)_{x_1}\,dx_1 +\right.$$$$\left.
 \int(|\nabla \Phi|^2\,-
\,2R_1(\Phi)+2l \Phi)_{x_2}\,dx_2 \right]=
$$
\begin{equation*} \label{pi_exact}
 \frac{1}{2c_0}\, \left(|\nabla \Phi|^2-2R_1(\Phi)-2l \Phi\right)\,+\,\rm
 const,
\end{equation*}
where $R_1=\int\limits_{\Phi_0}^\Phi R(\eta) d\eta.$

\end{proof}

\subsection{Relation with the eikonal equation}

\begin{proposition}\label{reducing_ei}
Assume that a function $\xi$ satisfies in a domain $\Omega$ the
standard eikonal equation
\begin{equation}\label{ei}
|\nabla \xi|^2=1.
\end{equation}
Then any differentiable monotone function $\Phi=\mathcal F(\xi)$
satisfies equation (\ref{ei_Phi}) with $G(\Phi)= (\Phi'(\xi))^2,$
$\xi=\mathcal F^{-1}(\Phi)$.
\end{proposition}
\begin{proof}
Proof is a direct computation.
\end{proof}


\section{Construction of localized frozen patterns}

\subsection{The smooth case}

Let $\Omega$ be a compact domain in ${\mathbb R}^2$ with smooth
boundary. We assume that there exist a couple of functions $G(\Phi),
R(\Phi))$ such that the stream function $\Phi$ satisfies equations
(\ref{ei_Phi}), (\ref{delta_Phi}) for $x\in \Omega$ in the classical
sense and
\begin{equation}\label{phi_CD}
\Phi|_{\partial \Omega}\,=\,0,\quad \nabla\Phi|_{\partial
\Omega}\,=\,0.
\end{equation}


\begin{proposition}\label{localized_pi1}
Let $\Phi\in C^2(\overline{\Omega})$ be a solution to
(\ref{ei_Phi}), (\ref{delta_Phi}) for $x\in \Omega$ with boundary
conditions (\ref{phi_CD}). Then
 the solution to (\ref{pi1}) in the domain ${\mathbb R}^2\setminus \Omega$
in the local coordinate system do not depend of the local field.
\end{proposition}
\begin{proof} Condition (\ref{phi_CD}) means  ${\bf u}|_{\partial \Omega}
=0$, therefore the velocity field can be extended smoothly to the
whole plane ${\mathbb R}^2$ as zero. Then the solution to
(\ref{pi1}) keeps its initial value for $x\notin \Omega$. $\square$
\end{proof}

Thus, the solution to problems  (\ref{ei_Phi}), (\ref{delta_Phi}),
(\ref{phi_CD}) gives a "frozen pattern" inside the domain $\Omega$.

Let us show that the class of solutions satisfying the conditions of
Proposition \ref{localized_pi1} is not empty.
 The simplest situation is where $\Omega$ is a disc of radius
$r$. The solution to (\ref{ei_CD_gen}) with boundary value $\phi=r$
is $\xi=\sqrt{x_1^2+x_2^2}$, the differentiability fails only at the
origin. Nevertheless, the solution to (\ref{phi_CD}) based on
(\ref{ei_CD_gen}) can be smooth everywhere in $\Omega$ if we take
$F(\xi)=\xi^2.$  Further, we can take  $F=\lambda(s)$, $s=|\xi^2|$,
with any smooth monotone on $[0,1]$ function $\lambda$ such that
$\lambda(r^2)=0,$ $\frac{d^i\lambda}{d s^i}\Big|_{s=r^2}=
\frac{d^i\lambda}{d s^i}\Big|_{s=0}=0, $ $i=1,2$. Further, $ \Delta
\Phi(\xi)=\Phi''(\xi)|\nabla \xi|^2 + \Phi'(\xi) \Delta \xi=
\Phi''(\xi)+ \frac{\Phi'(\xi)}{\xi}\equiv R(\xi). $

This example gives a variety of axisymmetric vortex structures.

\begin{remark} As follows from \cite{Padula}, Sec.2.3.3, these structures are
nonlinearly stable with respect to smooth perturbations keeping zero
boundary conditions.
\end{remark}

\subsection{Non-smooth case}


Now we consider non-classical solution to the system (\ref{ei_Phi}),
(\ref{delta_Phi}), allowing to construct the "frozen patterns"
containing discontinuities.

\subsubsection{Generalized solution to the eikonal equation (\ref{ei_Phi})}


It is well known that the problem of constructing a nonlocal theory
of the Cauchy problem for nonlinear Hamilton-Jacobi equation (in
particular, for (\ref{ei}) and (\ref{ei_Phi})) inevitably leads to
the necessity of introducing a generalized solution. A natural
extension of the notion of the solution in the sense of "almost
everywhere", i.e. a locally Lipschitz continuous function satisfying
the equation everywhere in the considered domain except possibly at
the points of a set of zero measure. Generally speaking, this
solution is not unique. In \cite{Kruzhkov} it was introduced the
following notion of generalized solution, which we formulate
applying to our case.

Let $B_r(y)=\{x\in{\mathbb R}^2:|x-y|<r   \}$. We  denote by
${\mathcal Lip}_{loc}(D)$ the totality of functions $f(x)$ defined
on a set $D\in {\mathbb R}^2$ and satisfying the Lipschitz condition
on the subset $B_r(y)\cap D$. It is known \cite{Evans} that $f(x)$
has at almost every interior point of $D$ a differential and hence a
gradient $\nabla u$.

Further, let $\Omega$ is a bounded domain in ${\mathbb R}^2$ and
$f\in {\mathcal Lip}_{loc}(\bar \Omega)$. We say that $f(x)$ belongs
to the stability class $E(\Omega)$ if the following inequality is
satisfied for every $x, x+\Delta x, x-\Delta x \in B_\delta
(y)\subset B_{2 \delta}\subset \Omega $, $(\Delta x \ne 0)$:
$$
\frac{\Delta^2 f}{|\Delta x|^2}\equiv \frac{f(x+\Delta x)-2
f(x)+f(x-\Delta x)}{|\Delta x|^2}\ge -C(y,\delta)=\rm const.
$$
\begin{definition} \label{4.1}A function $\Phi\in {\mathcal Lip}_{loc}(\bar \Omega)\cap E(\Omega)$
is called a generalized solution of the Cauchy-Dirichlet problem
\begin{equation}\label{ei_CD_gen}
|\nabla \Phi|^2=G(\Phi), \quad \xi|_{\partial \Omega}\,=\,\phi,
\end{equation}
where $G(\Phi)$ is a smooth and almost everywhere positive function,
if it satisfies equation (\ref{ei_Phi}) almost everywhere in
$\Omega$ and takes the boundary values.
\end{definition}

\subsubsection{Generalized solution to the nonlinear Poisson
equation (\ref{delta_Phi})} We give the definition of the
generalized solution following \cite{Evans_pde}.

\begin{definition}\label{4.2} $\Phi\in H^1(\Omega)$ is called the
generalized solution to the boundary problem
$$\Delta \Phi =R(\Phi),\quad x\in \Omega, \quad \Phi=\phi, \quad x\in \partial \Omega,$$
with  $R\in L_2(\Omega)$, $\phi\in H^1(\Omega),$ if
$$\Phi-\phi\in {H}_0^1(\Omega)$$
and
$$
\int\limits_{\Omega}(\nabla \Phi, \nabla v)
\,dx\,=\,-\int\limits_{\Omega}R(\Phi)  v \,dx
$$
for all $v\in {H}_0^1(\Omega).$
\end{definition}

\subsection{Example of discontinuous frozen pattern}

Now we construct a discontinuous solution to the system
(\ref{ei_Phi}), (\ref{delta_Phi}) in the sense of Definitions
\ref{4.1} and \ref{4.2}, satisfying zero boundary conditions.


We consider the domain $\Omega$ such that $(x_1,x_2)\in \Omega $ if
$x_1^2+x_2^2<1$ and $x_1<\frac12$.

It can be readily checked that the function
$$
\xi(x_1,x_2)=\left\{\begin{array}{cc} 2x_1,& -\sqrt{3} x_1<
x_2<\sqrt{3} x_1,\\ \sqrt{x_1^2+x_2^2},& otherwise,
\end{array}\right.
$$
solves the eikonal equation (\ref{ei}) in $\Omega $ with boundary
value $\xi=1$. Let us take again as $F=\lambda(s)$, $s=|\xi^2|$, any
monotone smooth function on $[0,1]$ such that $\lambda(1)=0,$
$\frac{d^i\lambda}{d s^i}\Big|_{s=r^2}= \frac{d^i\lambda}{d
s^i}\Big|_{s=0}=0, $ $i=1,2$.

Thus, $\Phi\in {\mathcal Lip}_{loc}(\bar \Omega)$ solves
(\ref{ei_Phi}) and (\ref{delta_Phi}) in the sense of Definition
\ref{4.2} with $R(\Phi)= \Phi''(\xi) +
\frac{\Phi'(\xi)}{\xi}(1-\chi(\Omega_1)),$ where $\chi(\Omega_1)$ is
the characteristic function of the set $\Omega_1$, where $\Omega_1$
is a subset of $\Omega$, consisting of points $(x_1,x_2)$ such that
$-\sqrt{3} x_1< x_2<\sqrt{3} x_1$.  Moreover, $\Phi$ satisfies the
boundary conditions (\ref{phi_CD}). Fig.1 presents the velocity
field for this case. The function $\lambda$ was chosen such that the
pressure $\pi_0$ in the center is lower that on the boundary of
$\Omega$, therefore the vorticity is cyclonic (anticlockwise). A
couple of lines of discontinuities  inside the vortical structure
can be used for modeling of cold and warm fronts in a cyclone.

\begin{figure}[h]
\centerline{\includegraphics[width=0.5\columnwidth]{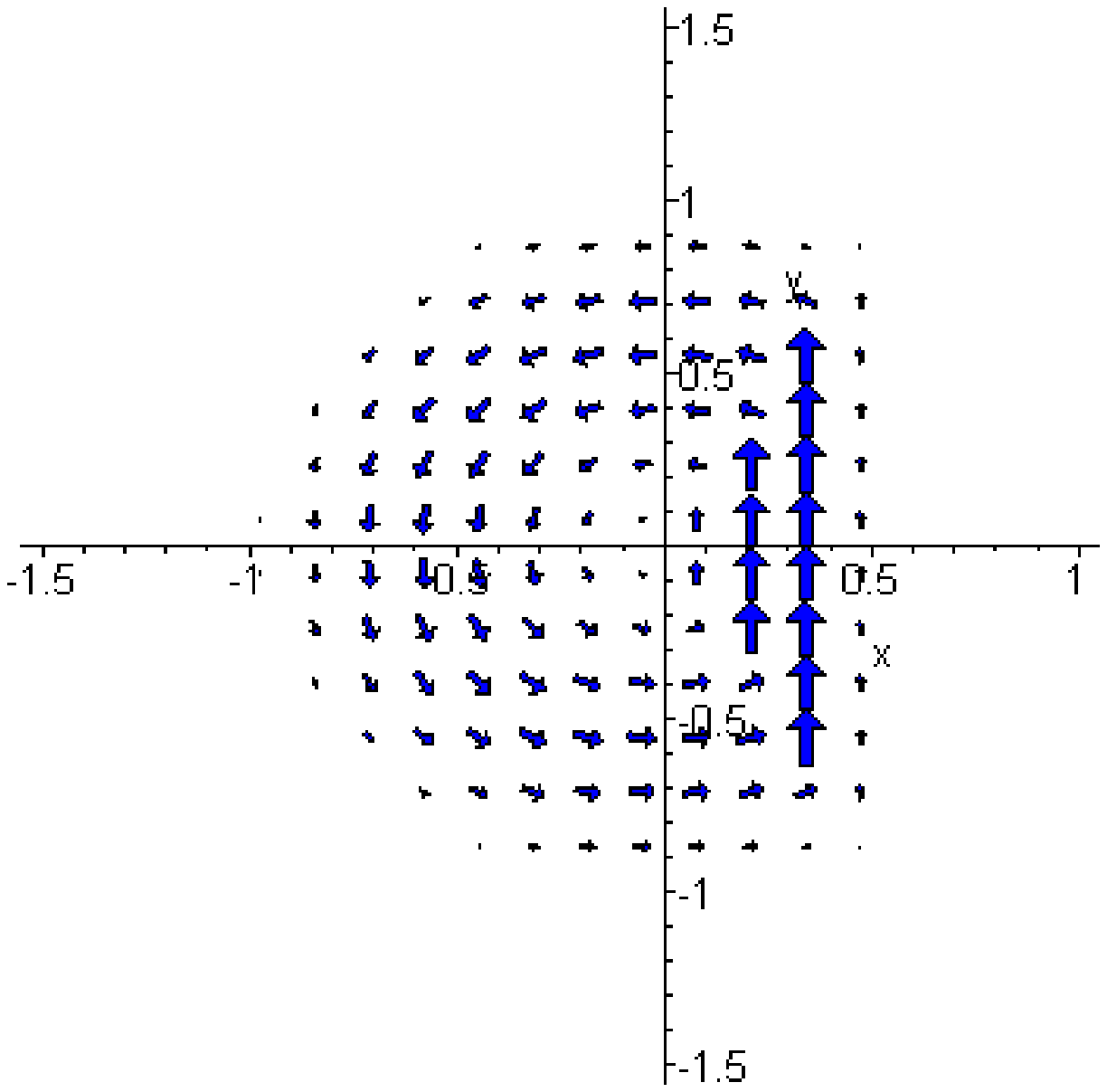}}
\caption{}\label{FFig1}
\end{figure}



\section{Estimate of the discrepancy}
\subsection{Smooth case}
\begin{proposition}
Let the stream function $\Phi$, the  solution to (\ref{ei_Phi}) and
(\ref{delta_Phi}), be differentiable in ${\mathbb R}^2$. Then the
discrepancy $Q$ (see (\ref{Q})) is completely defined by initial
data for the bearing field $\pi_1$.
\end{proposition}

\begin{proof} Since $ |Q|^2 \le 4 c_0^2 \sup\limits_{{\bf x}\in{\mathbb R}^2} |\nabla \pi_1(t,{\bf
x})|^2$, it is enough to prove that
\begin{equation}\label{sup}\sup\limits_{{{\bf x}\in \mathbb R}^2} |\nabla
\pi_1(t,{\bf x})|^2\le \sup\limits_{{{\bf x}\in \mathbb R}^2}
|\nabla \pi_1(0,{\bf x})|^2. \end{equation} Due to the
divergence-free condition equation (\ref{pi1}) has the form
$$
\partial_t \pi_1 +  (\nabla \pi_1 \cdot {\bf u})= 0.
$$
We take the gradient of this equation and then multiply by $\nabla
\pi_1$ to obtain the transport equation with smooth coefficients
$(u_1,u_2)$
\begin{equation}\label{transport}
\partial_t |\nabla \pi_1|^2 +  (\nabla |\nabla \pi_1|^2 \cdot {\bf u})= 0.
\end{equation}
Thus, the initial value are transported along smooth characteristic
curves. This immediately implies (\ref{sup}).
\end{proof}

\subsection{Non-smooth case}

Many papers are devoted to the transport equation with non-smooth
coefficients.  R.J. DiPerna and P.-L. Lions have proved this
uniqueness result under the assumption the coefficients are in the
Sobolev class $W^{1,1}$ (locally in space), and later L. Ambrosio
extended this result to coefficients of class $BV$ (locally in
space). Previous results on two-dimensional transport equation are
due to Bouchut and Desvillettes  \cite{Bouchut_Desvillettes}, Hauray
\cite{Hauray}, Colombini and Lerner \cite{Colombini}.

To show the maximum principle for (\ref{transport}) in the case on
non-smooth velocity ${\bf u}$, we use the following theorem on the
well-posedness of the transport equation proved in
\cite{Alberti_Bressan_Grippa1}, \cite{Alberti_Bressan_Grippa2}:

\begin{theorem}\label{th_abc}{\bf(\cite{Alberti_Bressan_Grippa1})} Let ${\bf b}: {\mathbb R}^2 \to {\mathbb R}^2$ be a bounded,
divergence-free, autonomous vector field on the plane admitting a
Lipschitz compactly supported potential $f: {\mathbb R}^2 \to
{\mathbb R}$, that is $\bf b=\nabla_\bot f$. The Cauchy problem for
$$
u_t + {\rm div} ({\bf b}u) = 0 $$ admits a unique bounded
generalized solution in the sense of distributions for every bounded
initial datum if and only if the potential $f$ satisfies the weak
Sard property.
\end{theorem}

\begin{corollary} Let the stream-function $\Phi$, taking part of the generalized solution to system
(\ref{constraint}), (\ref{Phi_pi_0}) be Lipschitz,  compactly
supported and $|\nabla \Phi|\ne 0$ almost everywhere. Then
\begin{equation*}\label{sup_gen} \|\nabla
\pi_1(t,{\bf x})\|_{L^\infty}\le \|\nabla \pi_1(0,{\bf
x})\|_{L^\infty}, \end{equation*} therefore the discrepancy is
completely defined by initial data for $\pi_1$.
\end{corollary}

\begin{proof} Ii is enough to write (\ref{transport}) in the divergent
form
\begin{equation*}\label{transport_div}
\partial_t |\nabla \pi_1|^2 +  {\rm div} (|\nabla \pi_1|^2 \nabla_\bot \Phi)= 0.
\end{equation*}
and apply Theorem \ref{th_abc}. As noticed in
\cite{Alberti_Bressan_Grippa1}, the weak Sard property is implied by
 $|\nabla \Phi|\ne
0$ almost everywhere.
\end{proof}

\end{document}